\documentclass[a4paper,10pt]{article}
\makeatletter
\usepackage[T1]{fontenc}
\usepackage[latin1]{inputenc}
\def\@tempa{:}
\if\@currdir\@tempa
\usepackage[applemac]{inputenc}
\fi
\DeclareRobustCommand{\qed}{%
  \ifmmode 
  \else \leavevmode\unskip\penalty9999 \hbox{}\nobreak\hfill
  \fi
  \quad\hbox{\qedsymbol}}
\newcommand{\emptysetpenbox}{\leavevmode
  \hbox to.77778em{%
  \hfil\vrule
  \vbox to.675em{\hrule width.6em\vfil\hrule}%
  \vrule\hfil}}
\newcommand{\qedsymbol}{\emptysetpenbox}
\newenvironment{proof}[1][\proofname]{\par
  \normalfont
  \topsep6\p@\@plus6\p@ \trivlist
  \item[\hskip\labelsep\itshape
    #1\@addpunct{.}]\ignorespaces
}{%
  \qed\endtrivlist
}

\newcommand{\proofname}{Proof}
\newenvironment{exemple*}[1][\examplename]{\par
  \normalfont
  \topsep6\p@\@plus6\p@ \trivlist
  \item[\hskip\labelsep\itshape\bfseries
     #1\@addpunct{.}]\ignorespaces\footnotesize
}{%
  \hfill $\clubsuit$\endtrivlist
}
\makeatother
\usepackage{enumerate}
\usepackage[centertags]{amsmath}
\usepackage[dvips]{graphics} 
\usepackage[T1]{fontenc}
\usepackage{amssymb}
\usepackage{pifont}
\usepackage{rotating}
\rotdriver{dvips}
\usepackage{theorem}
\usepackage{verbatim}
\author{P. Moyal \footnote{Laboratoire de Mathématiques Appliquées - 
Université de Technologie de Compiègne. 
Centre de Recherches de Royallieu, 
BP 20 529,  
60 205 COMPIEGNE Cedex, 
FRANCE.\,\,\,\, \emph{e-mail: moyalpas@dma.utc.fr}}\\
\small{\emph{LMAC - UTC Compiègne}}}
\title{Convex comparison of service disciplines in real time queues}
\newcommand\car{{\mathbf 1}}
\newcommand\N{{\mathbb N}}

\newcommand{\R}{{\mathbb R}}
\newcommand{\F}{{\mathcal F}}

\newcommand\esp[1]{{\mathbf E}\left[#1\right]}

\newcommand\espp[1]{{\mathbf E}^0\left[#1\right]}
\newcommand\prp[1]{{\mathbf P}^0\left[#1\right]}

\newcommand\pp{{\mathbf P^0}}

\newcommand\mpp{{\check{\mathbf P}^0}}

\def\N{{\mathbb N}}
\def\Z{{\mathbb Z}}

\newcommand\procR[1]{\left(#1_t\right)_{t \in \R}}

\newcommand\suiten[1]{\left\{#1\right\}_{n \in \N}}
\newcommand\suitez[1]{\left\{#1_n\right\}_{n \in \Z}}

{\bf}{\it}
\newtheorem{theorem}{Theorem}

\newtheorem{lemma}{Lemma}

{\bf}{\it}
\newtheorem{cor}{Corollary}{\bf}{\it}
{\theorembodyfont{\upshape}
{\theorembodyfont{\upshape\small}

}
\setcounter{tocdepth}{2}
\date{ }
\begin{document}
\maketitle
\date{ }

\begin{abstract}
We present a comparison of the service disciplines in real-time queueing systems (the customers have a deadline before which they should enter the service booth). We state that giving priority to customers having an early deadline minimizes the average stationary lateness. 
We show this result by comparing adequate random vectors with the Schur-Convex majorization ordering.  
\end{abstract}
\emph{keywords}: Queues with deadline; Earliest Deadline First; Stochastic comparison; Schur-convex ordering.\\\\ 
\section{Introduction}
In real-time queuing theory, each customer entering the queueing system is not only identified by his arrival time and service duration but also by a
\textsl{deadline}. In other words, a given customer has a given period of
time (his {patience}, {i.e.}, the remaining time before his deadline) during which he should enter the service booth. This patience then decreases at unit 
rate as time goes on. If the deadline is reached before the customer could enter service, the customer is either lost
and the deadline is said \emph{hard}, or he is kept in the waiting line and the deadline is said \emph{soft}. 
Hereafter we draw a comparison of the service disciplines in the soft deadline case. The performance metrics we focus on is the lateness 
of the customers with respect to their deadlines.

In both deterministic and random environments, it appears that the Earliest-Deadline-First (EDF
for short) discipline, which consists in giving priority to the customers having the nearest deadline, is optimal. 
It is the more feasible in the static case (see \cite{Dertou74}): if some preemptive discipline can serve the customers of a given scenario 
without any lateness, then preemptive EDF also does. It has been shown independently in \cite{LiuLay73} and \cite{SorHam75} that preemptive 
EDF is also optimal for job loss in the hard deadline context. More recently, 
\cite{StoGeo92} have proved that the soft preemptive EDF minimizes the lateness and tardiness of all the tasks in the system at a given time. 
In the hard deadline context and when randomness is assumed, non-preemptive EDF ensures the smallest loss for the strong stochastic ordering 
(see \cite{PanTow88}) and the least failure probability, i.e. it minimizes the stationary probability of being lost before service 
(\cite{moyal05}). It is shown in \cite{BhaEph89} that EDF also minimizes stochastically the number of lost customers among preemptive disciplines. 
In the soft deadline case, non-preemptive EDF also minimizes the largest lateness among all the customers 
in the system at a given time (or at a given finite set of times), for the strong stochastic ordering (see \cite{StoGeo92}).

In this paper, we consider the soft deadline case, and state that non-preemptive EDF also minimizes the mean stationary lateness 
of the customers. Our main result (Theorem \ref{thm:main}) is in fact more general, and establishes that the more the scheduler gives priority to customers who are close to their deadline, the smallest the mean expectation (in the Palm sense) of any convex function of the stationary lateness. 
The main mathematical tool used in the proof is the Schur-Convex comparison of random vectors in the Palm space of the arrival process. 
This framework is convenient for the comparison of service disciplines in other cases: the optimality of the SRPT discipline for the residual service vector (see \cite{Fli81}) and that of FIFO for the waiting time (see {e.g.} \cite{BacBre02}) have been shown in that sense. 
 
This paper is organized as follows. In section \ref{sec:order} we recall the main properties of the Schur-Convex majorization ordering. 
In section \ref{sec:prelim} we make precise the definitions, assumptions and notation on the G/GI/1 queueing station we consider. In section \ref{sec:result} we state and prove our main result. 
We conclude this paper with some comments in section \ref{sec:comments}.

\section{The majorization ordering on $\R^n$}
\label{sec:order}
Let us recall the definition and main properties of the \emph{majorization} partial semi-ordering on the Euclidian space $\R^n$. 
Let $\Gamma\left(1,...,n\right)$ denote the set of permutation of $\left\{1,...,n\right\}$ and for any vector $X=\left(X_1,X_2,...,X_n\right)$ and any $\gamma \in \Gamma\left(1,...,n\right)$, 
denote $X_{\gamma}:=\left(X_{\gamma(1)},X_{\gamma(2)},...,X_{\gamma(n)}\right)$. A function $F$ from $\R^n$ to some space $S$ is said symmetric whenever $F(X)=F\left(X_{\gamma}\right)$ for any $X \in \R^n$ and any $\gamma \in \Gamma\left(1,...,n\right)$. The \emph{fully ordered} version of any $X \in \R^n$ is $X_{\alpha}$ such that $X_{\alpha(i)}\le X_{\alpha(j)}$ for any $1\le i < j \le n$. We say that $\gamma \in \Gamma\left(1,...,n\right)$ is a \emph{reordering} permutation of $X \in \R^n$ if $X$ is not fully ordered and $\gamma(i)=j$ and $\gamma(j)=i$ for some $i<j$ such that $X_i > X_j$. 
  
Let $X,\,Y \in \R^n$ and $\alpha,\,\beta \in \Gamma\left(1,...,n\right)$. Denote $X_{\alpha}$ and $Y_{\beta}$ the fully ordered versions of $X$ and $Y$, respectively. Then we say that $X \prec Y$ if
\[\left\{\begin{array}{ll}
\displaystyle\sum_{i=1}^n X_i=\displaystyle\sum_{i=1}^n Y_i,&\\
\displaystyle\sum_{i=k}^n X_{\alpha(i)} \le \displaystyle\sum_{i=k}^n Y_{\beta(i)},&k=2,...,n.
\end{array}\right.\]
The properties of $\prec$ are thoroughly presented in \cite{Arnold87} and \cite{MarOl79}. 
Let us quote the following ones, which will be used in the sequel.
\begin{equation}
\label{eq:opposée}
X \prec Y \Longleftrightarrow -X \prec -Y,
\end{equation}
where e.g. $-X$ denotes the vector whose coordinates are the opposite of that of $X$, 
\begin{equation}
\label{eq:convex}
X \prec Y \Longleftrightarrow F(X)\le F(Y) \mbox{ for any convex symmetric real function }F.
\end{equation}
For any fully ordered $Z \in \R^n$ and any $\gamma$ that reorders $X$,   
\begin{equation}
\label{eq:lemma}
X_{\gamma}-Z \prec X - Z
\end{equation}
(see for instance \cite{BacBre02}, Lemma 4.1.1 p.266).

\section{The queue with smooth deadlines}
\label{sec:prelim}
Let $\left(\Omega,\F,\mathbf P,\theta_t\right)$ be a probability space furnished with a bijective flow $\procR{\theta}$, under which $\mathbf P$ is stationary and ergodic. Consider a single-server queueing station with infinite buffer, fed by a G/GI input: 
the arrival process 
$\procR{N}$ of points $...<T_{-1}<T_{0}\le 0 <T_1<T_2<...$ is compatible with $\procR{\theta}$ and hence its increments, or interarrivals 
$\xi_n:=T_{n+1}-T_n$, $n \in \Z$ form a stationary ergodic sequence of random variables (r.v. for short). The generic interarrival $\xi$ is assumed to be integrable. We denote for all $n \in \Z$, $C_n$ the $n$-th customer in the order of arrivals (i.e. entered at time $T_n$) and $\sigma_n$, the service duration requested by $C_n$. We assume furthermore that a smooth deadline $D_n$ is assigned to the customer $C_n$, before which he should begin his service. The \emph{initial patience} of $C_n$ (i.e. the time before his deadline) is hence given by $P_n:=D_n-T_n$. The sequences 
$\suitez{\sigma}$ and $\suitez{P}$ are sequences of marks of $\procR{N}$. Under these settings one can define the Palm space $\left(\Omega,\F,\pp, \theta\right)$ of $\procR{N}$, which is such that $\prp{T_0=0}=1$ and where $\theta:=\theta_{T_1}$ is a bijective stationary ergodic discrete flow. 
In particular, $\suitez{\sigma}$ and $\suitez{P}$ are stationary in that $\pp$-.a.s, $\sigma_n=\sigma_0\circ\theta^n$ and $P_n=P_0\circ\theta^n$, $n \in \Z$.  We assume furthermore that $\suitez{\sigma}$ is i.i.d., independent of $\procR{N}$ and $\suitez{P}$, and that the r.v. $\sigma:=\sigma_0$ and $P:=P_0$ are integrable. 
The deadlines are soft, thus any customer agrees to wait in the system as long as his service is not completed, even though he passed his deadline. 
Therefore the system is conservative (any entered customer is eventually served), and Loynes' stability condition for such a system, 
$$\espp{\sigma}/\espp{\xi}<1,$$ 
is assumed to hold. 
We denote by $W_n$, the waiting time of $C_n$ (the time he must wait before entering service), and $B_n=T_n+W_n$ the instant in which $C_n$ begins his service.
At any time $t \ge T_n$, the \emph{residual patience} (i.e. the remaining time before the deadline) of $C_n$ at time $t$ is given by $R_n(t)=D_n-t$. 
The residual patience of $C_n$ at the time he begins his service thus reads $R_n:=R_n\left(B_n\right)$.
Hence the \emph{lateness} (if any) of $C_n$ with respect to his deadline is given by 
$$L_n=\left(R_n\right)^-:=-\min\left(R_n,0\right).$$

Let $\Phi$ and $\Psi$ denote two admissible non-preemptive service discipline, i.e. two policies followed by the server when choosing the customer to serve, depending only on the state of the system at this instant (not on the already served customers, nor on the future customers to enter). We recall, that a discipline is said non-preemptive when the server always completes an initiated service even when a priority customer enters the system in the meanwhile. 
The compliance of the service discipline with respect to the real-time constraint is characterized as follows.  
We write $\Phi \ll \Psi$, whenever $\Phi$ always chooses a customer having a deadline earlier (and maybe already passed!) than that of the customer chosen by $\Psi$, in the case where both customers are present in the system and the service booth is free (i.e. at this instant $t$, if $\Phi$ chooses $C_i$ and $\Psi$ chooses $C_j$, then $R_i(t)\le R_j(t)$).    
In particular the \emph{Earliest Deadline First} (EDF) service discipline always gives priority to the customer having the earliest deadline, and 
the \emph{Latest Deadline First} (LDF) one gives priority to the customer having the latest deadline. Then, by definition
$$EDF \ll \Phi\ \ll LDF\mbox{ for any admissible discipline }\Phi.$$  
 

\section{Comparison of service disciplines}
\label{sec:result}
Let us emphasize for a given system the dependence on the service discipline by adding when necessary the superscript $^{\Phi}$ whenever the service discipline is $\Phi$. 
Since the stability condition is assumed to hold, their exists a stationary (i.e. compatible with $\theta$) waiting time $W^{\Phi}$ under discipline $\Phi$. Hence there exists a stationary residual patience when entering service and a stationary lateness, which are respectively given by
$$R^{\Phi}:=P-W^{\Phi}$$
and 
$$L^{\Phi}=\left(R^{\Phi}\right)^-.$$
We now state our main result. 

\begin{theorem}
\label{thm:main}
For any convex function $g:\R \to \R$, 
\begin{equation}
\label{eq:begin}
\espp{g\left(R^{\Phi}\right)} \le \espp{g\left(R^{\Psi}\right)}\mbox{ whenever } \Phi\ll \Psi.
\end{equation}
\end{theorem}

\begin{proof}
Let us assume that customer $C_0$ finds an empty queue upon arrival. We denote for any $k \ge 0$, $\phi(k)$ (resp. $\psi(k)$) the rank, in the increasing order of deadlines, of the $k$-th customer served under $\Phi$ (resp. $\Psi$). Let for any $j \ge 0$, $C_{\alpha(j)}$ be the $j$-th customer in the increasing order of deadlines (i.e. $D_{\alpha(i)}\le D_{\alpha(j)}$ for $i<j$),    
and for all $n\ge 0$, $\gamma(n)$ be such that the $n$-th customer served by $\Psi$ is the $\gamma(n)$-th customer served by $\Phi$. Since for all $k \ge 0$, $C_{\phi\circ\alpha(k)}$ (resp. $C_{\psi\circ\alpha(k)}$) is the $k$-th customer served under $\Phi$ (resp. $\Psi$), we have  
$$\gamma=\alpha\circ\psi\circ\phi^{-1}\circ\alpha^{-1}.$$
Let $N$ and $N^{\gamma}$ be the counting measures representing respectively the original double-marked input point process, and the input point process when the service times of the customers are rearranged following $\gamma$, i.e.
$$N:=\sum_{n \ge 0}\delta_{T_n,\sigma_n,P_n},$$
$$N^{\gamma}:=\sum_{n \ge 0}\delta_{T_n,\sigma_\gamma(n),P_n}.$$
Then, it is easily seen that
\begin{equation}
\label{eq:egalloi}
\mbox{$N$ and $N^{\gamma}$ have the same distribution}
\end{equation}  
(this is the "interchange argument" for the G/GI input, see \cite{BacBre02}, p.267). 
We add when necessary a superscript $^N$ (resp. $^{N^{\gamma}}$) whenever the input is $N$ (resp. $N^{\gamma}$).  
Let $T_{\tau}$ be the first time in which a customer enters an empty queue when starting with an empty queue just before time $T_0$. 
Note, that since the system is conservative, $T_{\tau}$ does not depend on the service discipline, nor on the order of services. 
We have for any $n \le \tau-1$
\begin{multline*}
B_{\alpha(n)}^{N^{\gamma},\Phi}=\sum_{i=1}^{\phi^{-1}(n)-1}\sigma_{\gamma\circ\alpha\circ\phi(i)}=\sum_{i=1}^{\phi^{-1}(n)-1}\sigma_{\alpha\circ\psi\circ\phi^{-1}\circ\alpha^{-1}\circ\alpha\circ\phi(i)}\\=\sum_{i=1}^{\phi^{-1}(n)-1}\sigma_{\alpha\circ\psi(i)}=B_{\gamma\circ\alpha(n)}^{N,\Psi},
\end{multline*} 
that is the instant in which customer $C_{\gamma\circ\alpha(n)}$ begins his service under $\Psi$. 
Thus 
we have 
\begin{equation}
\label{eq:alda1}
R_{\alpha}^{N^{\gamma},\Phi}=D_{\alpha}^{N^{\gamma}}-B_{\alpha}^{N^{\gamma},\Phi}=D_{\alpha}^{N^{\gamma}}-B_{\gamma\circ\alpha}^{N,\Psi}=D_{\alpha}^{N}-B_{\gamma\circ\alpha}^{N,\Psi}.
\end{equation}
Let us moreover remark that 
\begin{lemma}
\label{lemma:reorder}
$\gamma$ is a composition of reordering permutations of $B_{\alpha}^{N,\Psi}$.
\end{lemma}
\begin{proof}[Proof of Lemma \ref{lemma:reorder}]
The first integer $n$, if any, such that $\gamma(\alpha(n)) \neq \alpha(n)$ is such that at the $\alpha(n)$-th end of service under $\Psi$ (which is as well the $\alpha(n)$-th end of service under $\Phi$ since $\gamma(k)=k$ for $k=0,...,\alpha(n)-1$), there are in both systems (under $\Phi$ and $\Psi$) two customers, say $C_{i_1}$ and $C_{i_2}$ such that $D_{i_1}<D_{i_2}$, whereas the server takes care of $C_{i_1}$ under $\Phi$ and of $C_{i_2}$ under $\Psi$. In other words, 
denoting for $l=1,2$, $j_l=\alpha^{-1}(i_l)$, we have $B_{\alpha(j_2)}^{N,\Psi}<B_{\alpha(j_1)}^{N,\Psi}$ whereas $i_2=\alpha(j_2)>\alpha(j_1)=i_1$. Now, 
since $\Phi$ gives priority to $C_{i_1}$ over $C_{i_2}$, we have $\phi^{-1} (j_1)<\phi^{-1}(j_2).$ Hence from the definition of $\psi$, $$B_{\gamma\circ\alpha(j_1)}^{N,\Psi}=B_{\alpha\circ\psi\circ\phi^{-1}(j_1)}^{N,\Psi}<B_{\alpha\circ\psi\circ\phi^{-1}(j_2)}^{N,\Psi}=B_{\gamma\circ\alpha(j_2)}^{N,\Psi}.$$
Thus the permutation $\gamma_1$ interchanging $i$ and $j$ reorders $B_{\alpha}^{N,\Psi}$, and $\gamma$ reads\\ $\gamma=\gamma_p\circ...\circ\gamma_1,$ where the $\gamma_i$ are such permutations. 
\end{proof}
\noindent 
It is thus a consequence of Lemma \ref{lemma:reorder}, together with (\ref{eq:lemma}) and (\ref{eq:opposée}) that
$$D_{\alpha}^{N}-B_{\gamma\circ\alpha}^{N,\Phi}\prec D_{\alpha}^{N}-B_{\alpha}^{N,\Psi}=R_{\alpha}^{N,\Psi},$$
that is with (\ref{eq:alda1}),
$$R_{\alpha}^{N^{\gamma},\Phi}\prec R_{\alpha}^{N,\Psi}.$$
Hence from (\ref{eq:convex}), for any convex symmetric $F: \R^{\tau} \to \R$, 
$$F\left(R_{\alpha}^{N^{\gamma},\Phi}\right)\le F\left(R_{\alpha}^{N,\Psi}\right)$$
and in particular for all convex function $g:\R \to \R$, 
$$\sum_{i=1}^{\tau-1}g\left(R_{\alpha(i)}^{N^{\gamma},\Phi}\right)\le \sum_{i=1}^{\tau-1}g\left(R_{\alpha(i)}^{N,\Psi}\right).$$
Finally, taking expectations in this last inequaliy and in view of (\ref{eq:egalloi}),
\begin{multline*}
\esp{\sum_{i=1}^{\tau-1}g\left(R_i^{N,\Phi}\right)}=\esp{\sum_{i=1}^{\tau-1}g\left(R_{\alpha(i)}^{N,\Phi}\right)}\\\le \esp{\sum_{i=1}^{\tau-1}g\left(R_{\alpha(i)}^{N,\Psi}\right)}=\esp{\sum_{i=1}^{\tau-1}g\left(R_i^{N,\Psi}\right)}.
\end{multline*}
This leads to (\ref{eq:begin}) using the cycle formula of Palm probability. 
\end{proof}
\noindent
Since $g(x):=x^-$ is a convex function, we have 
\begin{cor}
\label{cor:late}
The mean lateness is minimized by EDF and maximized by LDF, i.e. for any admissible discipline $\Phi$, 
$$\espp{L^{\text{\tiny{EDF}}}}\le \espp{L^{\Phi}} \le \espp{L^{\text{\tiny{LDF}}}}.$$
\end{cor}

\section{Comments}
\label{sec:comments}

Let us first emphasize the differences between our results and that of \cite{StoGeo92}, Theorem 4 p.229. Corollary \ref{cor:late} states 
in particular that in the steady state, 
the lateness an arriving customer will undergo is (in average) minimized by EDF among all non-preemptive 
disciplines.  
Theorem 4 of \cite{StoGeo92} shows a complementary, but not directly related result: for a given $N$ and a given set of instants 
$t_1<t_2<...<t_N$, any increasing function (for the 
coordinatewise ordering) of the random vector $\left(Z(t_1),Z(t_2),....,Z(t_N)\right)$ is stochastically minimized 
by EDF, denoting for all $t\ge 0$, $Q(t)$ the number of customers in the systems at $t$ and 
$Z(t)=\max_{i=1}^{Q(t)} \left(-R_i(t)\right).$ In particular, this implies that for any $t\ge 0$, the maximal lateness 
(termed \emph{tardiness} in \cite{StoGeo92}) among the customers in the system at $t$ is stochastically minimized by EDF. The performance metrics 
there is the maximal lateness rather than the average lateness, as it is the case in our study.    

On another hand, remark that Corollary \ref{cor:late} is consistant with the optimality of the FIFO service discipline for the average waiting time (see e.g. \cite{BacBre02}, Property 4.1.3 p. 266), in the case where the initial patience of the customers is deterministic (say $P_n=p$, $\pp$-a.s.). Then, EDF reduces to FIFO (the first customer in the system has the earliest deadline), and the waiting time vector on a busy cycle reads $W=p-R$, where $p:=\left(p,p...,p\right)$.  Hence the Schur convex minimality of $W^{\text{\tiny{FIFO}}}$ amounts to that of $R^{\text{\tiny{FIFO}}}$, i.e. 
$R^{\text{\tiny{EDF}}}$. 


\end{document}